\documentclass[12pt,leqno]{amsart}
\newtheorem{theorem}{Theorem}[section]
\newtheorem{definition}{Definition}[section]

\newcommand{\be}{\begin{equation}}
\newcommand{\ee}{\end{equation}}
\numberwithin{equation}{section}
\newcommand{\bea}{\begin{eqnarray}}
\newcommand{\eea}{\end{eqnarray}}
\newcommand{\beb}{\begin{eqnarray*}}
\newcommand{\eeb}{\end{eqnarray*}}
\usepackage{amssymb,amsfonts,amsthm,setspace,indentfirst}
\usepackage[dvips]{graphics}
\usepackage{epsfig}
\begin{document}
\title{Rough $\mathcal{I}$-statistical convergence in a partial metric space}
\author{Sukila khatun$^{1}$, Khairul Hasan$^2$ and Amar Kumar banerjee$^{3}$}
\address{$^{1}$,$^{2}$,$^{3}$ Department of Mathematics, The University of Burdwan, Golapbag, Burdwan-713104, West Bengal, India.} 
\email{$^{1}$sukila610@gmail.com}
\email{$^{2}$khairul9734@gmail.com}
\email{$^{3}$akbanerjee@math.buruniv.ac.in, akbanerjee1971@gmail.com}
\begin{abstract}
 In this paper we study the notion of rough $\mathcal{I}$-statistical convergence of sequences in a partial metric space as an extension work of both the notions of rough statistical and rough ideal convergence. Here we define rough $\mathcal{I}$-statistical limit set and discuss some relevant properties associated with this set.
\end{abstract}
\maketitle
\section{\bf{Introduction }}
Using the notion of natural density of $\mathbb N$ the conception of ordinary convergence of real sequences were generalized to statistical convergence by H. Fast \cite{HF} and H. Steinhaus \cite{HS} independently. This notion of statistical convergence became an active research area in summability theory after the works of Fridy \cite{FRIDY} and Salat \cite{TS}. Further, using the concept of ideals of $\mathbb N$ the notion of statistical convergence were extended to $\mathcal{I}$-convergence by Kostyrko et al. \cite{PK3} (see also \cite{PK2}). Lot of work on $\mathcal{I}$-convergence or $\mathcal{I}$-divergence can be found in \cite{{PM4}, {PM5}, {AKBAP1}, {AKBAP2}, {AKBMP}}. In 2011, the notion of ideal statistical convergence in short $\mathcal{I}$-statistical convergence was given by P. Das et al. \cite{PD2} as a generalization of the notions of statistical convergence and $\mathcal{I}$-convergence. Again, Yamanc and Gurdal \cite{YG} studied the concept of ideal statistical convergence in 2-normed spaces. 
The idea of rough convergence of sequences in a normed linear space was introduced by H. X. Phu \cite{PHU,PHU1} in 2001. Thereafter more works on rough convergence were carried out in various directions (see also \cite{{RMROUGH}, {DR}, {PMROUGH1},{SUK1}}). Further, using the concept of natural density of $\mathbb N$ the idea of rough convergence was extended to rough statistical convergence by S. Ayter \cite{AYTER1}. More investigations of this idea can be found in \cite{{SUK4}, {PMROUGH2}} and many others. Again, the idea of rough convergence was extended to rough $\mathcal{I}$-convergence by Pal et al.\cite{PAL}. Further investigations on this work were done in \cite{{NH}, {SUK6}}. P. Malik et al. \cite{PMROUGH3} studied the idea of rough $\mathcal{I}$-statistical convergence of sequences in a normed linear space.
In 1994, Matthews \cite{MATW1} (see also \cite{MATW2}) introduced the concept of partial metric spaces as a generalization of metric spaces, using the notion of self-distance $d(x,x)$ which may not be zero where as in a metric space it is always zero. In our present work we discuss the idea of rough $\mathcal{I}$-statistical convergence of sequences in a partial metric space. Also we have found out several properties of rough $\mathcal{I}$-statistical limit set and other relevant properties of this.

\section{\bf{Preliminaries}}

    If $ B \subset \mathbb{N}$, then $B_{n}$ will denote the set $ \{k \in B : k \leq n \}$ and $|B_{n}|$ stands for the cardinality of $B_{n}$. The natural density of $B$ is denoted by $d(B)$ and defined by $d(B)=lim_{n\to\infty} \frac{|B_{n}|}{n}$, if the limit exists.\\ 
    A real sequence $\{ \xi_n \}$ is said to be statistically convergent to $\xi$ if for every $\varepsilon>0$ the set 
    \begin{center}
        $B(\varepsilon)= \{ k \in \mathbb{N}: |\xi_{k} - \xi| \geq \varepsilon \}$
    \end{center}
    has natural density zero. In this case, $\xi$ is called the statistical limit of the sequence $\{ \xi_n \}$ and we write $st-lim \ \xi_{n}= \xi$. We have $d(B^c)=1-d(B)$, where $B^c=\mathbb{N} \setminus B$ is the complement of $B$. If $B_1 \subset B_2$, then clearly, $d(B_1) \leq d(B_2)$.\\
Let $X$ be a non-empty set. We denote $2^X$ as the power set of $X$.
Then a family of sets $\mathcal{I} \subset 2^X$ is said to be an ideal if $(i) \phi \in \mathcal{I}$ 
$(ii) A,B \in \mathcal{I} \Rightarrow A \cup B \in \mathcal{I}$ 
$(iii) A \in \mathcal{I}, B \subset A \Rightarrow B \in \mathcal{I}$. \\ 
$\mathcal{I}$ is called non-trivial ideal if $\mathcal{I} \neq 2^X,\{\phi\}$. A non-trivial ideal $\mathcal{I}$ in $X$ is called admissible if $\{x\} \in \mathcal{I}$ for each $x \in X$. Clearly the family $\mathcal{F(I)}=\{A \subset X: X \setminus A \in \mathcal{I}\}$ is a filter on $X$ which is called the filter associated with $\mathcal{I}$.

\begin{definition} \cite{PK1}
    Let $\mathcal{I}$ be a non-trivial ideal of $\mathbb N$. A sequence $\{x_n\}$ in $\mathbb R$ is said to be $\mathcal{I}$-convergent to $x$ if for every $\varepsilon>0$, the set $A(\varepsilon)=\{n \in \mathbb N: |x_n-x| \geq \varepsilon\} \in \mathcal{I}$. \\
    If $\{x_n\}$ is $\mathcal{I}$-convergent to $x$, then $x$ is called $\mathcal{I}$-limit of $\{x_n\}$and we write $\mathcal{I}-lim \ x_n=x$.
\end{definition}

\begin{definition} \cite{PK1}
    Let $\mathcal{I}$ be an admissible ideal in $\mathbb N$. A sequence $\{x_n\}$ of real numbers is said to be $\mathcal{I^*}$-convergent to $x$ (shortly $\mathcal{I^*}-lim \ x_n=x$) if there is a set $M=\{m_1 < m_2 < ... \}\in \mathcal {F(I)}$ such that $lim_{k\to\infty}x_{m_k}=x$.
\end{definition}

\begin{definition}\cite{BMW}
A partial metric on a non-empty set $X$ is a function $p: X\times X \longrightarrow [0, \infty)$ such that for all $x,y,z \in X$:\\
$(p1)$ $0 \leq p(x,x) \leq p(x,y)$ (nonnegativity and small self-distances),\\
$(p2)$ $x=y \Longleftrightarrow p(x,x)=p(x,y)=p(y,y)$ (indistancy both implies equality),\\
$(p3)$  $p(x,y)= p(y,x)$ (symmetry),\\
$(p4)$  $p(x,y) \leq p(x,z) + p(z,y) - p(z,z)$ (triangularity).\\
If $p$ is a partial metric on a nonempty set $X$, then the pair $(X,p)$ is said to be a partial metric space.

Properties and examples of partial metric spaces were widely discussed in \cite{BMW}.
\end{definition}

\begin{definition} \cite{BMW}
In a partial metric space $(X,p)$, for $r>0$ and $x \in X$ we define the open and closed ball of radius $r$ and center $x$ respectively as follows  :
\begin{center}
    $B^{p}_{r}(x)=\{ y \in X : p(x,y)<p(x,x)+r  \}$ \\
$\overline{B^{p}_{r}}(x)=\{ y \in X : p(x,y) \leq p(x,x)+r \}.$ 
\end{center}
\end{definition}

\begin{definition} \cite{BMW}
Let $(X,p)$ be a partial metric space. A subset $U$ of $X$ is said to be a bounded in $X$ if there exists a positive real number $M$ such that $sup$ $\{ p(x,y): x,y \in U\}< M$.
\end{definition}

\begin{definition} \cite{BMW}
Let $(X,p)$ be a partial metric space and $\{x_{n}\}$ be a sequence in $X$. Then $\{x_{n}\}$ is said to converge to $x \in X$ if for each $\epsilon > 0$ there exists $k \in \mathbb{N}$ such that 
$|p(x_{n},x)-p(x,x)|< \epsilon$ for all $n \geq k$, i.e., if $lim_{n\to\infty}p(x_{n},x)=p(x,x)$. 
\end{definition}

\begin{definition} \cite{SUK2}
Let $(X,p)$ be a partial metric space. A sequence $\{ x_{n} \}$ in $X$ is said to be rough convergent (or $r$-convergent) to $x$ of roughness degree $r$ for some non-negative real number $r$ if, for every  $\epsilon > 0$, there exists a natural number $k$ such that $| p(x_{n}, x)-p(x,x) |<r + \epsilon $ holds for all $n \geq k$.
\end{definition}

\begin{definition}\cite{FN}
    Let $(X,\rho)$ be a partial metric space. Then the sequence $\{x_n\}$ is said to be statistically convergent to $x$ if, for every $\varepsilon>0$,
    \begin{center}
        $d(\{n \in \mathbb N: |\rho(x_{n},x)-\rho(x,x)| \geq \varepsilon\})=0$.
    \end{center}
\end{definition}

\begin{definition} \cite{SUK4}
    A sequence $\{x_n\}$ in a partial metric space $(X,\rho)$ is said to be rough statistically convergent (or in short $r$-statistically convergent or $r-st$ convergent) to $x$ of roughness degree $r$ for some non-negative real number $r$ if, for every $\varepsilon>0$, 
    \begin{center}
        $d(\{n \in \mathbb N : |p(x_n,x)-p(x,x) | \geq r+\varepsilon\})=0$.
    \end{center}
\end{definition}

\begin{definition} \cite{DUN}
Let  $\mathcal{I}$ be a non-trivial admissible ideal of $\mathbb N$. A sequence $\{x_n\}$ in a partial metric space $(X,p)$ is said to be ideal convergent ($\mathcal{I}$-convergent) to $x \in X$ if for every $\varepsilon>0$, the set $A(\varepsilon)=\{n \in \mathbb N: |p(x_{n},x)-p(x,x)| \geq \varepsilon\} \in \mathcal{I}$ i.e., if $\mathcal{I}-lim_{n\to\infty}p(x_{n},x)=p(x,x)$.
\end{definition}

\begin{definition} \cite{SUK6}
Let $(X,p)$ be a partial metric space and $\{x_n\}$ be a sequence in $X$. Then the sequence $\{x_n\}$ is said to be rough ideal convergent (or in short rough $\mathcal{I}$-convergent or $r-\mathcal{I}$ convergent) to $x$ if, for any $\varepsilon>0$, the set 
\begin{center}
    $A(\varepsilon)=\{n \in \mathbb N : |p(x_n,x)-p(x,x) | \geq r+\varepsilon\} \in \mathcal{I}$.
\end{center}
\end{definition}

\section{\bf{Rough ideal statistical convergence in partial metric spaces}}

\begin{definition}
A sequence $\{\xi_n\}$ in a partial metric space $(X,p)$ is said to be rough ideal statistically convergent (or in short rough $\mathcal{I}$-statistically convergent or $r-\mathcal{I}$-statistically convergent) to $\xi$ if for any $\varepsilon>0$ and $\delta>0$, the set 
\begin{center}
    $A=\{n \in \mathbb N : \frac{1}{n}|\{k \leq n: |p(\xi_k,\xi)-p(\xi,\xi)|  \geq r+\varepsilon \} | \geq \delta\} \in \mathcal{I}$.
\end{center}
\end{definition}

It will be denoted by  $\xi_{n} \stackrel{r-\mathcal{I}-st}{\longrightarrow} \xi$ in $(X,p)$. Here $r$ is said to be the degree of roughness. If $r=0$, then the rough $\mathcal{I}$-statistical convergence becomes the $\mathcal{I}$-statistical convergence in any partial metric space $(X,p)$. If a sequence $\{\xi_n\}$ is rough $\mathcal{I}$-statistically convergent to $\xi$, then $\xi$ is said to be a rough $\mathcal{I}$-statistical limit point. The set of all rough $\mathcal{I}$-statistical limit points of a sequence $\{\xi_n\}$ is said to be the rough $\mathcal{I}$-statistical limit set. This can be denoted by $\mathcal{I}-st-LIM^{r}\xi_{n}$. So, $\mathcal{I}-st-LIM^{r}\xi_{n}= \left\{\xi \in X : \xi_{n} \stackrel{r-\mathcal{I}-st}{\longrightarrow} \xi \right\}$. \\  

\begin{theorem}
    For any sequence $\{\xi_n\}$, $r-st-lim \ p(\xi_n,\xi)=p(\xi,\xi)$ implies $r-\mathcal{I}-st-lim \ p(\xi_n,\xi)=p(\xi,\xi)$.
\end{theorem}
\begin{proof} 
 Let $r-st-lim \ p(\xi_n,\xi)=p(\xi,\xi)$.\\
 Then for $\varepsilon>0$, the set $A=\{ k \in \mathbb N: |p(\xi_{k},\xi)-p(\xi,\xi)| \geq r+\varepsilon \}$ has natural density zero.\\
 i.e., $lim_{n\to\infty} \frac{1}{n}|\{k \leq n: |p(\xi_{k},\xi)-p(\xi,\xi)| \geq r+\varepsilon\}|=0$.\\
 So, for every $\varepsilon>0$ and $\delta>0$,\\
 $\{n \in \mathbb N : \frac{1}{n}|\{k \leq n: |p(\xi_k,\xi)-p(\xi,\xi)|  \geq r+\varepsilon \} | \geq \delta\}$ is a finite set and therefore belongs to $\mathcal{I}$, as $\mathcal{I}$ is an admissible ideal.
 Hence $r-\mathcal{I}-st-lim \ p(\xi_n,\xi)=p(\xi,\xi)$. 
\end{proof}

\begin{definition} \cite{SUK2}
    The diameter of a set $B$ in a partial metric space $(X,p)$ is defined by 
    \begin{center}
    $diam(B)$ = $sup$ $\{ p(x,y) : x, y\in B\}$.
\end{center}
\end{definition}

\begin{theorem}
Let $(X,p)$ be a partial metric space and $a$ be a fixed positive real number such that $p(\xi,\xi)=a$ for all $\xi$ in $X$. Then for a sequence $\{\xi_{n}\}$, we have $diam(\mathcal{I}-st-LIM^r\xi_n) \leq (2r+2a)$. 
\end{theorem}

\begin{proof}
Let us suppose that $diam(\mathcal{I}-st-LIM^{r} \xi_{n}) > 2r+2a$.
Then there exist elements $\alpha,\beta \in \mathcal{I}-st-LIM^{r} \xi_{n}$ such that $p(\alpha,\beta)>2r+2a$.
Let $\varepsilon>0$ and take $\varepsilon =\frac{p(\alpha,\beta)}{2}-r-a$. 
Let $K_1=\{ k \in \mathbb N: |p(\xi_{k},\alpha)-p(\alpha,\alpha)| \geq r+\varepsilon \}$ and $K_2=\{ k \in \mathbb N: |p(\xi_{k},\beta)-p(\beta,\beta)| \geq r+\varepsilon \}$.
Now, \begin{equation*}
    \frac{1}{n}|\{k \leq n: k \in K_1 \cup K_2\}| \leq \frac{1}{n}|\{k \leq n: k \in K_1\}|+ \frac{1}{n}|\{k \leq n: k \in K_2\}|
\end{equation*}
Again, from the property of $\mathcal{I}$-convergence \\
$\mathcal{I}-lim_{n\to\infty} \frac{1}{n}|\{k \leq n: k \in K_1 \cup K_2\}| \leq \mathcal{I}-lim_{n\to\infty} \frac{1}{n}|\{k \leq n: k \in K_1\}|+ \mathcal{I}-lim_{n\to\infty} \frac{1}{n}|\{k \leq n: k \in K_2\}|=0.$
Thus $\{ n \in \mathbb N: \frac{1}{n}|\{k \leq n: k \in K_1 \cup K_2\}| \geq \delta \} \in \mathcal{I}$ for $\delta>0$.
Let $K=\{ n \in \mathbb N: \frac{1}{n}|\{k \leq n: k \in K_1 \cup K_2\}| \geq \frac{1}{n}\}$.
Clearly $K \in \mathcal{I}$ and we choose $n_0 \in \mathbb N \setminus K$.
Then $\frac{1}{n_0}|\{k \leq n_0: k \in K_1 \cup K_2\}| < \frac{1}{3}$.
So, $\frac{1}{n_0}|\{k \leq n_0: k \notin K_1 \cup K_2\}| \geq 1- \frac{1}{3}= \frac{2}{3}$ i.e., $\{k: k \notin K_1\cup K_2\}$ is a non-empty set.
Take $k_0 \in \mathbb N$ such that $k_0 \notin K_1 \cup K_2$.
So, $k_0 \in K_1^c \cap K_2^c$.
Hence $|p(\xi_{k_0},\alpha)-p(\alpha,\alpha)| \leq r+\varepsilon$ and 
$|p(\xi_{k_0},\beta)-p(\beta,\beta)| \leq r+\varepsilon$.
Now, 
 \begin{equation*}
       \begin{split}
p(\alpha,\beta) &\leq p(\alpha,\xi_{k_0})+p(\xi_{k_0},\beta)-p(\xi_{k_0},\xi_{k_0})\\
       &=\{p(\xi_{k_0},\alpha)-p(\alpha,\alpha)\}+\{\xi_{k_0},\beta)-p(\beta,\beta)\}-p(\xi_{k_0},\xi_{k_0})+p(\alpha,\alpha)+p(\beta,\beta)\\
       &< 2(r+\varepsilon)-a+a+a\\
       &=2r+2\varepsilon+a\\
       &<2r+ p(\alpha,\beta)-2r-2a+a\\
       &=p(\alpha,\beta)-a , \ \text{which is a contradiction}.    
       \end{split}
   \end{equation*}    
Therefore, we have $diam(\mathcal{I}-st-LIM^{r} x_{n}) \leq 2r+2a$. 
\end{proof}

\begin{theorem}
 If a sequence $\{\xi_{n}\}$ is  $\mathcal{I}$-statistically convergent to $\xi$ in a partial metric space $(X,p)$, then $\{ \eta \in \overline{B^{p}_{r}}(\xi):p(\xi,\xi)=p(\eta,\eta)\} \subset \mathcal{I}-st-LIM^{r}\xi_{n}$.
\end{theorem}

\begin{proof}
 Let $\varepsilon >0$ and $\delta>0$ be given and a sequence $\{\xi_{n}\}$ be $\mathcal{I}$-statistically convergent to $\xi$ i.e.,$\mathcal{I}-st-lim\ {\xi_n}=\xi$.
 Then the set $A=\{n \in \mathbb N: \frac{1}{n} |\{ k \leq n: |p(\xi_{k},\xi)-p(\xi,\xi)| \geq \varepsilon \}|\geq \delta\} \in \mathcal{I}$.
 For $n \notin A$, we have $ \frac{1}{n} |\{ k \leq n: |p(\xi_{k},\xi)-p(\xi,\xi)| \geq \varepsilon \}|< \delta$, i.e.,  $\frac{1}{n} |\{ k \leq n: |p(\xi_{k},\xi)-p(\xi,\xi)| < \varepsilon \}|> 1-\delta$ .....(1) \\
 Now $\overline{B^{p}_{r}}(x)=\{\eta \in X: p(\xi,\eta) \leq p(\xi,\xi)+r\}$.\\ So if $\eta \in\{\eta \in \overline{B^{p}_{r}}(\xi):p(\xi,\xi)=p(\eta,\eta)\}$, then 
 \begin{equation*}
    \begin{split}
        p(\xi_{k},\eta) &\leq p(\xi_{k},\xi)+p(\xi,\eta)-p(\xi,\xi)\\
                   & = \{p(\xi_{k},\xi)-p(\xi,\xi)\}+p(\xi,\eta)\\
                   & \leq \{p(\xi_{k},\xi)-p(\xi,\xi)\} + \{p(\xi,\xi)+r\}
    \end{split}
\end{equation*}
 Therefore, 
\begin{equation*}
    \begin{split}
p(\xi_{k},\eta)-p(\eta,\eta) & \leq \{p(\xi_{k},\xi)-p(\xi,\xi)\} + \{p(\xi,\xi)+r-p(\eta,\eta)\}\\
                  & =\{p(\xi_{k},\xi)-p(\xi,\xi)\} + r, \ \ \text{since} \ p(\xi,\xi)=p(\eta,\eta) \\
                  &=|p(\xi_k,\xi)-p(\xi,\xi)|+r.....(2)
    \end{split}
\end{equation*}
Let $B_n=\{ k \leq n: |p(\xi_{k},\xi)-p(\xi,\xi)| < \varepsilon \}$.
Then for $k \in B_n$ we have, from (2)\\ $|p(\xi_{k},\eta)-p(\eta,\eta)|=|p(\xi_k,\xi)-p(\xi,\xi)|< r+\varepsilon$.\\
Hence $B_n \subset \{ k \leq n: |p(\xi_{k},\eta)-p(\eta,\eta)| < r+\varepsilon \}$.\\
So, $\frac{|B_n|}{n} \leq \frac{1}{n}|\{ k \leq n: |p(\xi_{k},\eta)-p(\eta,\eta)| < r+\varepsilon \}|$ \\
From (1) we have $\frac{|B(n)|}{n}>1-\delta$,\\
so, $\frac{1}{n} |\{ k \leq n: |p(\xi_{k},\eta)-p(\eta,\eta)| < r+\varepsilon\}|\geq\frac{|B(n)|}{n} > 1-\delta$.\\
So, for all $n \notin A$, \\
$\frac{1}{n} |\{ k \leq n: |p(\xi_{k},\eta)-p(\eta,\eta)| \geq r+\varepsilon\}| < 1-(1-\delta)=\delta$.\\
Thus we have 
$\{n \in \mathbb N: \frac{1}{n} |\{ k \leq n: |p(\xi_{k},\eta)-p(\eta,\eta)| \geq r+\varepsilon\}| \geq \delta\} \subset A$.
Since $A \in \mathcal{I}$, so 
$\{n \in \mathbb N: \frac{1}{n} |\{ k \leq n: |p(\xi_{k},\eta)-p(\eta,\eta)| \geq r+\varepsilon\}| \geq \delta\} \in \mathcal{I}$.
Therefore, $\eta \in \mathcal{I}-st-LIM^{r}\xi_{n}$.
Hence the theorem.
\end{proof}

\begin{theorem}
   Let $\{\xi_{n}\}$ be a sequence in a partial metric space $(X,p)$. Then the rough $\mathcal{I}$-statistical limit set of the sequence $\{\xi_{n}\}$ i.e., the set $\mathcal{I}-st-LIM^{r}\xi_{n}$ is a closed set for any degree of roughness $r \geq 0$.
\end{theorem}

\begin{proof}
If $\mathcal{I}-st-LIM^{r}\xi_{n}=\phi$, then there is nothing to prove.
So let $\mathcal{I}-st-LIM^{r}\xi_{n} \neq \phi$.
Now, we consider a sequence $\{\eta_{n}\}$ in $\mathcal{I}-st-LIM^{r}\xi_{n}$ such that $lim_{n\to\infty}\eta_{n}=\eta$.
Choose $\varepsilon>0$ and $\delta>0$.
Then there exists $N_\frac{\varepsilon}{2} \in \mathbb N$ such that $|p(\eta_k,\eta)-p(\eta,\eta)| < \frac{\varepsilon}{2}$ for all $k \geq N_\frac{\varepsilon}{2}$.
Let $k_0 \geq N_\frac{\varepsilon}{2}$.
Then $|p(\eta_{k_0},\eta)-p(\eta,\eta)|<\frac{\varepsilon}{2}$. Since $\eta_{k_0} \in \mathcal{I}-st-LIM^{r}\xi_{n}$, so 
$A=\{n \in \mathbb N: \frac{1}{n} |\{ k \leq n: |p(\xi_{k},\eta_{k_0})-p(\eta_{k_0},\eta_{k_0})| \geq r+\frac{\varepsilon}{2} \}|\geq \delta\} \in \mathcal{I}$.
So, $M=\mathbb N \setminus A$ is non-empty.\\
If $n \in M$, then
$\frac{1}{n} |\{ k \leq n: |p(\xi_{k},\eta_{k_0})-p(\eta_{k_0},\eta_{k_0})| \geq r+\frac{\varepsilon}{2} \}| < \delta$ \\
i.e., $\frac{1}{n} |\{ k \leq n: |p(\xi_{k},\eta_{k_0})-p(\eta_{k_0},\eta_{k_0})| < r+\frac{\varepsilon}{2} \}|> 1-\delta$.\\
Put $B_n=\{ k \leq n: |p(\xi_{k},\eta_{k_0})-p(\eta_{k_0},\eta_{k_0})| < r+\frac{\varepsilon}{2} \}$.
Let $k \in B_n$. Then 
\begin{equation*}
    \begin{split}
p(\xi_{k},\eta)-p(\eta,\eta) & \leq p(\xi_{k},\eta_{k_0})+p(\eta_{k_0},\eta)-p(\eta_{k_0},\eta_{k_0})-p(\eta,\eta) \\
         & \leq | p(\xi_{k},\eta_{k_0})-p(\eta_{k_0},\eta_{k_0})+p(\eta_{k_0},\eta)-p(\eta,\eta) | \\
         & \leq |p(\xi_{k},\eta_{k_0})-p(\eta_{k_0},\eta_{k_0})| + |p(\eta_{k_0},\eta)-p(\eta,\eta)| \\
         & < (r+\frac{\varepsilon}{2}) + \frac{\varepsilon}{2} \\
         & = r+\varepsilon
    \end{split}
\end{equation*}
Hence $B_n \subset \{ k \leq n: |p(\xi_{k},\eta)-p(\eta,\eta)| < r+\varepsilon\}$. \\
Thus if $n\in M=\mathbb N \setminus A$ , then
$1-\delta < \frac{|B_n|}{n} \leq \frac{1}{n}|\{ k \leq n: |p(\xi_{k},\eta)-p(\eta,\eta)| < r+\varepsilon \}|$ \\
i.e., $\frac{1}{n} |\{ k \leq n: |p(\xi_{k},\eta)-p(\eta,\eta)| \geq r+\varepsilon\}| < 1-(1-\delta)=\delta$.\\
So, we have 
$\{n \in \mathbb N: \frac{1}{n} |\{ k \leq n: |p(\xi_{k},\eta)-p(\eta,\eta)| \geq r+\varepsilon\}| \geq \delta\} \subset A \in \mathcal{I}$.\\
This shows that 
$\eta \in \mathcal{I}-st-LIM^{r}\xi_{n}$.
Hence $\mathcal{I}-st-LIM^{r}\xi_{n}$ is a closed set.
\end{proof} 

\begin{definition} (cf \cite{AKB})
   A sequence $\{\xi_{n}\}$ in a partial metric space $(X,p)$ is said to be $\mathcal{I}$-statistically bounded if for any fixed $u \in X$ there exists a positive real number $M$ such that for any $\delta>0$ the set 
   \begin{center}
       $A=\{ n \in \mathbb N : \frac{1}{n}|\{ k \leq n:p(\xi_n,u) \geq M\}| \geq \delta\} \in \mathcal{I}$.
   \end{center} 
\end{definition}

\begin{theorem}
Let $(X,p)$ be a partial metric space and $a$ be a fixed positive real number such that $p(\xi,\xi)=a$, $ \forall \xi \in X$. Then a sequence $\{\xi_{n}\}$ is $\mathcal{I}$-statistically bounded if and only if there exists a non-negative real number $r>0$ such that $\mathcal{I}-st-LIM^{r}\xi_{n} \neq \phi$.     
\end{theorem}

\begin{proof}
Let $\{\xi_{n}\}$ be a $\mathcal{I}$-statistically bounded sequence.
Then for any fixed element $u \in X$
 there exists a positive real number $M$ such that for $\delta>0$, we have 
 \begin{center}
     the set  $A=\{ n \in \mathbb N : \frac{1}{n}|\{ k \leq n:p(\xi_n,u) \geq M\}| \geq \delta\} \in \mathcal{I}$.
 \end{center}
 Let $\varepsilon>0$ be arbitrary and $r=M+a$.
Clearly $B=\mathbb N \setminus A$ is non-empty. Choose $n \in B$. Then we have 
\begin{center}
    $\frac{1}{n}|\{ k \leq n:p(\xi_n,u) \geq M\}| < \delta$ \\
    i.e., $\frac{1}{n}|\{ k \leq n:p(\xi_n,u) < M\}| > 1-\delta$
\end{center}
Put $C_n=\{ k \leq n:p(\xi_n,u) < M\}$. Choose $k \in C_n$. 
Then \begin{equation*}
    \begin{split}
          |p(\xi_k,u)-p(u,u)| & \leq |p(\xi_k,u)|+|p(u,u)| \\
        & < M+a =r <r+\varepsilon.
    \end{split}
\end{equation*}
Hence $C_n \subset \{ k \leq n: |p(\xi_k,u)-p(u,u)| <r+\varepsilon \}$.
This implies that
\begin{center}
    $1-\delta<\frac{|C_n|}{n} \leq \frac{1}{n} |\{ k \leq n: |p(\xi_n,u)-p(u,u)| <r+\varepsilon \}|$ \\
i.e., $\frac{1}{n} |\{ k \leq n: |p(\xi_n,u)-p(u,u)| \geq r+\varepsilon \}| < 1-(1-\delta)=\delta$.
\end{center}
Thus we have 
$\{ n \in \mathbb N : \frac{1}{n}|\{ k \leq n: |p(\xi_k,u)-p(u,u)| \geq r+\varepsilon\}| \geq \delta\} \subset A \in \mathcal{I}$.
This shows that $u \in \mathcal{I}-st-LIM^{r}\xi_{n}$ i.e., $\mathcal{I}-st-LIM^{r}\xi_{n} \neq \phi$. \\

Conversely, let $\mathcal{I}-st-LIM^{r}\xi_{n} \neq \phi$  and let  $u \in \mathcal{I}-st-LIM^{r}\xi_{n}$ . 
Then for any $\varepsilon >0$ and $\delta>0$, the set
$ A=\{ n \in \mathbb N : \frac{1}{n}|\{ k \leq n: |p(\xi_n,u)-p(u,u)| \geq r+\varepsilon\}| \geq \delta\} \in \mathcal{I}$. 
Since $\mathcal{I}$ is non trivial the set $B=\mathbb N \setminus A$ is non-empty. Then for $n \in B$ we have,
\begin{center}
    $\frac{1}{n}|\{ k \leq n: |p(\xi_n,u)-p(u,u)| \geq r+\varepsilon\}| < \delta$ \\
    i.e., $\frac{1}{n}|\{ k \leq n: |p(\xi_n,u)-p(u,u)| < r+\varepsilon\}| \geq 1-\delta$
\end{center}
Put $D_n=\{ k \leq n: |p(\xi_n,u)-p(u,u)| < r+\varepsilon\}$. Choose $k \in D_n$. Then 
\begin{equation*}
    \begin{split}
        p(\xi_k,u) & = |p(\xi_k,u)-p(u,u)+p(u,u)| \\
                & \leq |p(\xi_k,u)-p(u,u)| + |p(u,u)| \\
                & < r+a+\varepsilon = M(say).
    \end{split}
\end{equation*}
So $D_n \subset \{ k \leq n: p(\xi_k,u) < M\}$.
This implies that $1-\delta \leq \frac{|D_n|}{n} \leq \frac{1}{n} |\{ k \leq n: p(\xi_k,u) < M\}|$ \\
i.e., $\frac{1}{n} |\{ k \leq n: p(\xi_k,u) \geq M \}| < 1-(1-\delta)=\delta$.
Thus we have 
$\{ n \in \mathbb N : \frac{1}{n}|\{ k \leq n: p(\xi_k,u) \geq M \}| \geq \delta\} \subset A \in \mathcal{I}$.
Hence $\{\xi_{n}\}$ is $\mathcal{I}$-statistically bounded. 
\end{proof}

\begin{definition}\cite{PMROUGH3}
A subset $K$ of $\mathbb N$ is said to have $\mathcal{I}$-natural density $d_\mathcal{I}(K)$ if
\begin{center}
    $d_\mathcal{I}(K)=\mathcal{I}-lim_{n\to\infty}\frac{|K(n)|}{n}$ exists
\end{center}
 where $K(n)=\{{j\in {K}:j\leq n\}}$ and $|K(n)|$ represents the number of elements in $K(n)$.
\end{definition}

\begin{theorem}
    Let $\{\xi_{n}\}$ be a sequence in $X$ and $r>0$ be any real number. If $\{\xi_{n}\}$ has a subsequence $\{ \xi_{{i}_{k}}\}$ 
    satisfying the condition $\mathcal{I}-lim_{n\to\infty}\frac{1}{n}|\{i_k \leq n: k \in \mathbb{N}\}|=1$, then $\mathcal{I}-st-LIM^{r}\xi_{n} \subseteq \mathcal{I}-st-LIM^{r}\xi_{i_{k}}$. 
\end{theorem}

\begin{proof}
Let $\xi \in \mathcal{I}-st-LIM^{r}\xi_{n}$.
Let $\varepsilon>0$ be given.
Since $\xi \in \mathcal{I}-st-LIM^{r}\xi_{n}$, we have
$\mathcal{I}-lim_{n\to\infty}\frac{1}{n} |\{ k \leq n: |p(\xi_{k},\xi)-p(\xi,\xi)| \geq r+\varepsilon\}|=0$, i.e.,
$\mathcal{I}-lim_{n\to\infty}\frac{1}{n} |\{ k \leq n: |p(\xi_{k},\xi)-p(\xi,\xi)|< r+\varepsilon\}|=1$.....(1).
Again by the given condition we have 
$\mathcal{I}-lim_{n\to\infty}\frac{1}{n}|\{i_k \leq n: k \in \mathbb{N}\}|=1$.....(2).
Let $A=\{i_k: k \in \mathbb{N}\}$. Then by (2) we have $d_\mathcal{I}(A)=1$, where $d_{\mathcal{I}}(A)=\mathcal{I}-lim_{n\to\infty} \frac{1}{n}|\{k \leq n: k \in A\}|$ and so, $d_\mathcal{I}(\mathbb{N} \setminus A)=0$.
Then from (1), we have 
$\mathcal{I}-lim_{n\to\infty}\frac{1}{n} |\{ i_k \leq n: |p(\xi_{k},\xi)-p(\xi,\xi)|< r+\varepsilon\}|=1$.....(3).
Now, $1 \geq \frac{1}{n} |\{k \leq n: |p(\xi_{i_k},\xi)-p(\xi,\xi)|< r+\varepsilon\}|=\frac{1}{n} |\{i_k \leq i_n: |p(\xi_{i_k},\xi)-p(\xi,\xi)|< r+\varepsilon\}|$.....(4).
Also, $\frac{1}{n} |\{i_k \leq i_n: |p(\xi_{i_k},\xi)-p(\xi,\xi)|< r+\varepsilon\}| \geq \frac{1}{n} |\{i_k \leq n: |p(\xi_{i_k},\xi)-p(\xi,\xi)|< r+\varepsilon\}|$.....(5).
Thus from (3), (4), (5) and by the property of $\mathcal{I}$-convergence, we have 
$\mathcal{I}-lim_{n\to\infty}\frac{1}{n} |\{k \leq n: |p(\xi_{i_k},\xi)-p(\xi,\xi)|< r+\varepsilon\}|=1$, i.e.,
$\mathcal{I}-lim_{n\to\infty}\frac{1}{n} |\{k \leq n: |p(\xi_{i_k},\xi)-p(\xi,\xi)|\geq r+\varepsilon\}|=0$.
Hence $\{\xi_{i_k} \}$ is $r-\mathcal{I}-$statistical convergent to $\xi$ i.e., $\xi \in \mathcal{I}-st-LIM^{r}\xi_{i_{k}}$.
Therefore $\mathcal{I}-st-LIM^{r}\xi_{n} \subseteq \mathcal{I}-st-LIM^{r}\xi_{i_{k}}$.
\end{proof}

\begin{theorem}
Let $(X, p)$ be a partial metric space and let $\{a_{n}\}$ and $\{b_{n}\}$ be two sequences such that $p(a_{n}, b_{n}) \longrightarrow 0$ as $n \longrightarrow \infty$.
Then $\{a_{n}\}$ is rough $\mathcal{I}$-statistically convergent to $a$ and $p(a_{n},a_{n}) \longrightarrow 0$ as $n \longrightarrow \infty$ if and only if $\{b_{n}\}$ is rough $\mathcal{I}$-statistically convergent to $a$. 
\end{theorem}

\begin{proof}
First suppose that $\{a_{n}\}$ be rough $\mathcal{I}$-statistically convergent to $a$. Let $\varepsilon>0$.
Then for $\varepsilon>0$ and any $\delta>0$ the set 
\begin{center}
    $A=\{n \in \mathbb N : \frac{1}{n}|\{k \leq n: |p(a_k,a)-p(a,a)|  \geq r+\frac{\varepsilon}{3} \} | \geq \delta\} \in \mathcal{I}$.
\end{center}
Since $p(a_{n},b_{n}) \longrightarrow 0$ as $n \longrightarrow \infty$, for $\varepsilon>0$ there exists $m_1 \in \mathbb N$ such that
\begin{equation} \label{a}
     p(a_{n},b_{n})\leq \frac{\varepsilon}{3}, \ \text{whenever} \ n \geq m_1.
 \end{equation}

Again, since $p(a_{n},a_{n}) \longrightarrow 0$ as $n \longrightarrow \infty$ there exists $ m_2 \in \mathbb N$ such that 
\begin{equation} \label{b}
    p(a_{n},a_{n})\leq \frac{\varepsilon}{3}, \ \text{whenever} \ n \geq m_2.
\end{equation}
Let $ m= \ max \{ m_1, m_2 \}$. Then equation (3.1) and (3.2) both hold for $n \geq m$. \\
Let $n \notin A$. Then we have \\ 
$\frac{1}{n}|\{k \leq n: |p(a_n,a)-p(a,a)|  \geq r+\frac{\varepsilon}{3} \} | < \delta$ \\
i.e., $\frac{1}{n}|\{k \leq n: |p(a_n,a)-p(a,a)| < r+\frac{\varepsilon}{3} \} | \geq 1-\delta$.\\
Let $B_n=\{k \leq n: |p(a_n,a)-p(a,a)| < r+\frac{\varepsilon}{3} \}$. 
Then for $k \in B_n$, we have \\
 \begin{equation*}
    \begin{split}
 p(b_{k},a) &  \leq p(b_{k},a_{k})+p(a_{k},a)-p(a_{k},a_{k}) \\
 p(b_{k},a) - p(a,a) & \leq p(b_{k},a_{k})+p(a_{k},a)-p(a_{k},a_{k}) - p(a,a) \\ |p(b_{k},a) - p(a,a)|=p(b_k,a)-p(a,a)
    & \leq |p(b_{k},a_{k}) + p(a_{k},a) - p(a_{k},a_{k}) - p(a,a)| \\
    & \leq |p(b_{k},a_{k})| + |p(a_{k},a) - p(a,a)| + |p(a_{k},a_{k})|\\
    & < \frac{\varepsilon}{3} + (r+ \frac{\varepsilon}{3})+\frac{\varepsilon}{3} \\
    &  = r + \varepsilon 
    \end{split}
\end{equation*}
Hence $B_n \subset \{k \leq n: |p(b_{k},a) - p(a,a)|< r+\varepsilon\}$. This implies that 
\begin{center}
    $\frac{|B_n|}{n} \leq \frac{1}{n} |\{ k \leq n: |p(b_{k},a) - p(a,a)| <r+\varepsilon \}|$ \\
i.e., $\frac{1}{n} |\{ k \leq n: |p(b_{k},a) - p(a,a)| <r+\varepsilon \}| \geq (1-\delta)$ \\
So, $\frac{1}{n} |\{ k \leq n: |p(b_{k},a) - p(a,a)| \geq r+\varepsilon \}| < 1-(1-\delta)=\delta$.
\end{center}
Hence $\{n \in \mathbb N : \frac{1}{n}|\{k \leq n: |p(b_k,a)-p(a,a)|  \geq r+\varepsilon\}| \geq \delta\} \subset A$.
Since $A \in \mathcal{I}$, we have
$\{n \in \mathbb N : \frac{1}{n}|\{k \leq n: |p(b_k,a)-p(a,a)|  \geq r+\varepsilon\}| \geq \delta\} \in \mathcal{I}$.
So,  $\{b_{n}\}$ is rough $\mathcal{I}$-statistically convergent to $a$. \\
Converse part is similar.
\end{proof}

\begin{theorem}
Let $\{a_{n}\}$ and $\{b_{n}\}$ be two sequences in a partial metric space $(X, p)$ such that $p(a_{n},b_{n}) \longrightarrow 0$ as $ n \longrightarrow \infty$. 
If $\{a_{n}\}$ is rough $\mathcal{I}$-statistically convergent to $a$ of roughness degree $r$ and if $c$ is a positive number such that $ p(a_{n},a_{n}) \leq c$ for all $n$, then $\{b_{n}\}$ is rough $\mathcal{I}$-statistically convergent to $a$ of roughness degree $r+c$. 
Conversely, if $\{b_{n}\}$ is rough $\mathcal{I}$-statistically convergent to $b$ of roughness degree $r$ and $d$ is a positive number such that $p(b_{n},b_{n})\leq d$ for all $n$, then $\{a_{n}\}$ is rough $\mathcal{I}$-statistically convergent to $b$ of roughness degree $r+d$.      
\end{theorem}

\begin{proof}
The proof is parallel to the proof of the above theorem and so is omitted. 
\end{proof}

\begin{definition}
Let $(X, p)$ be a partial metric space. Then a point $c \in X$ is said to be an $\mathcal{I}$--statistical cluster point of a sequence $\{\xi_{n}\}$ if for every $\varepsilon>0$ 
\begin{center}
    $d_{\mathcal{I}}(\{k: |p(\xi_k,c)-p(c,c)| < \varepsilon \}) \neq 0$, if exists,\\
\end{center}
where $d_{\mathcal{I}}(A)=\mathcal{I}-lim_{n\to\infty} \frac{1}{n}|\{k \leq n: k \in A\}|, A=\{k: |p(\xi_k,c)-p(c,c)| < \varepsilon \}$.
\end{definition}

\begin{theorem}
 Let $\{\xi_{n}\}$ be a sequence in a partial metric space $(X, p)$ and $p(\xi,\xi)=a$ for all $\xi \in X$, where $a$ be a real constant. If $c$ is a $\mathcal{I}$-statistical cluster point of $\{\xi_{n}\}$, then $\mathcal{I}-st-LIM^{r}\xi_{n}  \subset \overline{B^{p}_{r}}(c)$ for $r >0$.
\end{theorem}

\begin{proof}
If possible, let there exists $\xi \in \mathcal{I}-st-LIM^{r}\xi_{n}$ such that $p(\xi,c)>p(c,c)+r$.
Let $\varepsilon=\frac{p(\xi,c)-p(c,c)-r}{2}$.
Then we claim that $\{ k \in \mathbb{N}: |p(\xi_{k},\xi)-p(\xi,\xi)| \geq r+\varepsilon\} \supset \{k \in \mathbb{N}: |p(\xi_k,c)-p(c,c)| < \varepsilon \}$.....(1).
For, let $k \in \{k \in \mathbb{N}: |p(\xi_k,c)-p(c,c)| < \varepsilon \}$. \\ 
Then $|p(\xi_k,c)-p(c,c)| < \varepsilon \implies -|p(\xi_k,c)-p(c,c)| >- \varepsilon$.\\
Now, \begin{equation*}
    \begin{split}
        p(\xi,c) & \leq p(\xi_k,\xi)+p(\xi_k,c)-p(\xi_k,\xi_k) \\  \implies p(\xi,c) & \leq p(\xi_k,\xi)+p(\xi_k,c)-p(\xi,\xi), \text{as} \ p(\xi,\xi)=a \ \forall \ \xi \in X \\
 \implies p(\xi_k,\xi)-p(\xi,\xi) & \geq p(\xi,c)-p(\xi_k,c) \\
                       & = p(\xi,c)-p(\xi_k,c)-p(c,c)+p(c,c) \\
                       & = p(\xi,c)-p(c,c)-\{p(\xi_k,c)-p(c,c)\} \\
                       &= p(\xi,c)-p(c,c)-|p(\xi_k,c)-p(c,c)| \\
                       & \geq 2\varepsilon + r- \varepsilon \\
                       &  =r+\varepsilon \\
    \end{split}
\end{equation*}
So, $|p(\xi_k,\xi)-p(\xi,\xi)| \geq r+\varepsilon$.                     Therefore, $k \in \{ k \in \mathbb{N}: |p(\xi_{k},\xi)-p(\xi,\xi)| \geq r+\varepsilon\}$. \\
Since $c$ is a $\mathcal{I}$-statistical cluster point of $\{\xi_{n}\}$, so
$d_{\mathcal{I}}(\{k \in \mathbb{N}: |p(\xi_k,c)-p(c,c)| < \varepsilon \}) \neq 0$.
Hence by (1) $d_{\mathcal{I}}(\{ k \in \mathbb{N}: |p(\xi_{k},\xi)-p(\xi,\xi)| \geq r+\varepsilon\}) \neq 0$, which contradicts $\xi \in \mathcal{I}-st-LIM^{r}\xi_{n}$.
Hence $p(\xi,c) \leq p(c,c)+r \implies \xi \in \overline{B^{p}_{r}}(c)$.
Therefore, $\mathcal{I}-st-LIM^{r}\xi_{n}  \subset \overline{B^{p}_{r}}(c)$.
\end{proof}

\subsection*{Acknowledgements}
Authors are also thankful to DST, Govt. of India, for providing the FIST project to the Department of Mathematics, B.U. \\

\end{document}